\documentclass{agtart_a}
\pdfoutput=1

\usepackage{amscd}

\title{Cohomology of preimages with local coefficients}

\author{Daciberg Goncalves} 
\givenname{Daciberg}
\surname{Goncalves}
\address{Dept. de Matematica\\
IME - USP\\\newline
Caixa Postal 66.281\\
CEP 05311-970 Sao Paulo - SP\\Brasil}
\email{dlgoncal@ime.usp.br}
\urladdr{}

\author{Peter Wong}
\givenname{Peter}
\surname{Wong}
\address{Department of Mathematics\\
Bates College\\\newline
Lewiston, ME 04240\\USA}
\email{pwong@bates.edu}
\urladdr{}

\volumenumber{6}
\issuenumber{}
\publicationyear{2006}
\papernumber{54}
\startpage{1471}
\endpage{1489}

\doi{}
\MR{}
\Zbl{}


\received{16 March 2005}
\revised{25 April 2006}
\accepted{14 August 2006}
\published{4 October 2006}
\publishedonline{4 October 2006}
\proposed{}
\seconded{}
\corresponding{}
\editor{}
\version{}

\arxivreference{} 

\let\xysavmatrix\xymatrix
\def\xymatrix{\disablesubscriptcorrection\xysavmatrix}
\AtBeginDocument{\let\bar\wbar\let\tilde\wtilde\let\hat\what}
\makeop{sgn}

\makeatletter
\def\cnewtheorem#1[#2]#3{\newtheorem{#1}{#3}[section]
\expandafter\let\csname c@#1\endcsname\c@theorem}
\makeatother

\newtheorem{theorem}{Theorem}[section]
\cnewtheorem{lemma}[theorem]{Lemma}
\cnewtheorem{proposition}[theorem]{Proposition}
\cnewtheorem{corollary}[theorem]{Corollary}

\theoremstyle{definition}
\cnewtheorem{definition}[theorem]{Definition}
\cnewtheorem{example}[theorem]{Example}
\numberwithin{equation}{section}
\newtheorem*{remark}{Remark}

\begin{document}

\begin{asciiabstract}
Let M,N and B\subset N be compact smooth manifolds of dimensions n+k,n
and \ell, respectively. Given a map f from M to N, we give homological
conditions under which g^{-1}(B) has nontrivial cohomology (with local
coefficients) for any map g homotopic to f. We also show that a
certain cohomology class in H^j(N,N-B) is Poincare dual (with local
coefficients) under f^* to the image of a corresponding class in
H_{n+k-j}(f^{-1}(B)) when f is transverse to B. This generalizes a
similar formula of D\,Gottlieb in the case of simple coefficients.
\end{asciiabstract}

\begin{htmlabstract}
Let M,N and B&sub; N be compact smooth manifolds of dimensions
n+k,n and <i>l</i>, respectively. Given a map f:M &rarr;
N, we give homological conditions under which g<sup>-1</sup>(B) has
nontrivial cohomology (with local coefficients) for any map g
homotopic to f. We also show that a certain cohomology class in
H<sup>j</sup>(N,N-B) is Poincar&eacute; dual (with local coefficients) under
f<sup>*</sup> to the image of a corresponding class in H<sub>n+k-j</sub>(f<sup>-1</sup>(B))
when f is transverse to B. This generalizes a similar formula of
D&nbsp;Gottlieb in the case of simple coefficients.
\end{htmlabstract}

\begin{abstract}
Let $M,N$ and $B\subset N$ be compact smooth manifolds of dimensions
$n+k,n$ and $\ell$, respectively. Given a map $f\colon\thinspace M \to
N$, we give homological conditions under which $g^{-1}(B)$ has
nontrivial cohomology (with local coefficients) for any map $g$
homotopic to $f$. We also show that a certain cohomology class in
$H^j(N,N{-}B)$ is Poincar\'e dual (with local coefficients) under
$f^*$ to the image of a corresponding class in $H_{n+k-j}(f^{-1}(B))$
when $f$ is transverse to $B$. This generalizes a similar formula of
D\,Gottlieb in the case of simple coefficients.
\end{abstract}

\maketitle

\section{Preliminaries}

In 1978, D\,Gottlieb \cite{GWt} proved a beautiful formula which, in particular for fibrations, 
relates the homology
fundamental class of the fiber
with the cohomology fundamental class of the base. More precisely, let $f: M \to N$ be a smooth map where $M$ and $N$ are closed orientable smooth manifolds. Let $y\in N$ be a regular value for $f$ and 
let $F=f^{-1}(y)$ and consider $i_1\co F\hookrightarrow M$. If $[z_F]$ and $[z_M]$ denote the homology fundamental classes of $F$ and of $M$ respectively, and $\mu_N$ denotes the cohomology fundamental class of $N$, then Gottlieb \cite{GWt} proved that
\begin{equation}\label{gottlieb}
f^*(\mu_N) \frown [z_M]=i_{1_*}([z_F]).
\end{equation}
According to A\,Dold \cite{GWt}, this fact ``is a special case of a vague principle known to Hopf, which says that if a cycle and a cocycle are related by Poincar\'e duality in $N$, then $f^{-1}(\text{cycle})$ and $f^*(\text{cocycle})$ should be related by Poincar\'e duality in $M$." In fact, a similar relation was established by H\,Samelson in \cite{samelson} as follows. Suppose $F\subset M$ and $M$ are closed connected oriented manifolds of dimensions $k$ and $n+k$ respectively, and $\nu$ and $\mu$ are the corresponding homology fundamental classes of $F$ and of $M$. Then it was shown in Theorem B of \cite{samelson} that 
\begin{equation}\label{sam}
i_{1_*}(\nu)=u_F \frown \tilde \mu
\end{equation}
where $i_1\co F\hookrightarrow M$ is the inclusion, $u_F \in H^n(M,M-F)$ is the Thom class of $F$ in $M$ and $\tilde \mu \in H_{n+k}(M,M-F)$ is the element determined by $\mu$. The map 
$f\co M\to N$ induces an isomorphism
$H^k(N)\cong H^k(N,N-y_0) \to H^k(M,M-F)$ in the setting of \cite{GWt}, and one can deduce \eqref{gottlieb} from \eqref{sam}. 

Gottlieb used the formula \eqref{gottlieb} to
obtain results about the transfer \cite{GWt} and about the trace of a group action \cite{GWt2}. A similar formula was also obtained for Poincar\'e duality groups by Daccach and Franco \cite{da-fr}.
Local coefficients have played an important role in the study of nonorientable manifolds and in obstruction theory (see also Borsari and Gon\c calves \cite{bo-go}). In particular, 
a certain type of Gottlieb's formula has been obtained and used  in \cite{GJW} in the study of the coincidence problem of maps between manifolds of different dimensions. Let $p,q\co X\to Y$ be maps between two closed orientable manifolds such that $p$ and $q$ are transverse so that $C(p,q)=\{x\in X\,|\,p(x)=q(x)\}$ is a submanifold of dimension $m-n$ where $m=\dim X\ge \dim Y=n$. Under conjugation, $\mathbb Z [\pi]$ is a local system on $Y\times Y$ where $\pi=\pi_1(Y)$. The first author, Wong and Jezierski defined a twisted Thom class $\tilde \tau_Y \in H^n(Y^{\times};\mathbb Z [\pi])$ \cite{GJW}. Here $Y^{\times}$ denotes the pair $(Y\times Y, Y\times Y - \Delta Y)$ where $\Delta Y$ is the diagonal of $Y$ in $Y\times Y$. The primary obstruction to deforming $p$ and $q$ to be coincidence free on the $n$--th skeleton of $X$ is defined by 
$$
o_n(p,q):=[j(p\times q)d]^{*}(\tilde \tau_N)
$$
where $d\co X\to X\times X$ is the diagonal map and $j\co Y\times Y \hookrightarrow Y^{\times}$ is the inclusion. Then it was shown in \cite{GJW} that $o_n(p,q)$ is Poincar\'e dual to the twisted fundamental class representing $C(p,q)$ with local coefficients $\pi^*$ induced from $\mathbb Z[\pi]$, ie,
\begin{equation}\label{gjw-formula}
o_n(p,q) \frown [z_X] = [z^{\pi^*}_{C(p,q)}] \quad \in H_{m-n}(X;\pi^*).
\end{equation}
Furthermore, this formula has been employed successfully in Gon{\c{c}}alves and Wong \cite{daci-peter2} in the study of coincidences and of roots. While Gottlieb's formula \eqref{gottlieb} relates the cohomology fundamental class of the base 
$N$ with the homology fundamental class of the fiber $F$, it does not determine whether $F$ is empty or not and hence it does not give any information about the (homological) size of $F=f^{-1}(y)$.

To see how cohomology with local coefficients can yield more information, let us consider the following example.

\begin{example}\label{gjw-ex}
Let $\mathcal H$ be the Heisenberg $3$--dimensional nilmanifold, $T$ the $2$--dimen\-sion\-al torus and $$S^1 \stackrel{i}{\to} \mathcal H \stackrel{f}{\to} T$$ a fibration inducing the homomorphism in the fundamental group $f_{\#}\co \pi_1(\mathcal H)  \to \pi_1(T)$ which is the abelianization. To avoid technical difficulties in dimension 2, consider the  fibration $p\co \mathcal H\times S^1\to T\times S^1=T^3$ where $p=f\times 1_{\smash{S^1}}$ is the product map. The homomorphism $i_*\co H_1(S^1) \to H_1(\mathcal H\times S^1)$ induced by the inclusion of the fiber is the zero homomorphism since $H_1(\mathcal H\times S^1) \to H_1(T^3)$ is an isomorphism. By Gottlieb's formula, we have
$$
i_*(\mu'_1)=p^*(z_2)\frown \mu_1
$$
where $\mu_1$ denotes the homology fundamental class of $\mathcal H\times S^1$, $z_2$ the cohomology fundamental class of $T^3$ and $\mu'_1$ the homology fundamental class of the circle fiber $S^1=p^{-1}(y), y\in T^3$. 
This shows that $p^*(z_2)=0$.

On the other hand, using the local coefficient system $\mathbb Z[\pi_2]$ on $\mathcal H\times S^1$ induced by the map $p$, one can show that $H^3(\pi_1;\mathbb Z[\pi_2])\cong \mathbb Z\ne 0$ and in addition that the inclusion $S^1 \hookrightarrow \mathcal H\times S^1$ induces an isomorphism on homology with local coefficients. Here, $\pi_1=\pi_1(\mathcal H\times S^1)$ and $\pi_2=\pi_1(T^3)$.
Because the target is a group, a root problem can be regarded as a coincidence problem (see \fullref{coin=root}). Then we can apply  
\eqref{gjw-formula} and see that the cohomology class representing the primary obstruction to deforming the map to be root free is Poincar\'e dual to the twisted fundamental class $$[z^{\pi^*_2}_{p^{-1}(y)}]\in H_1(\pi_1;\mathbb Z[\pi_2])\cong H^3(\pi_1;\mathbb Z[\pi_2]).$$ This class turns out to be nontrivial \cite{daci-peter2}. It follows that the obstruction to deforming $p$ to $\hat p$ with $\hat p^{-1}(y)=\emptyset$ is nontrivial. Hence, $F$ has nontrivial homology (with local coefficients) in dimension one.
\end{example}

This example uses two important facts: (i) the primary obstruction to deforming a map to be root free is Poincar\'e dual to the twisted fundamental class by means of a Gottlieb type formula, and (ii) the homomorphism induced by $p$ is nontrivial when local coefficients are employed. Another example which illustrates the phenomenon above can be found in \cite{GJW}.

In studying injective points, A\,Wyler \cite{wyler} proved an interesting result concerning the cohomological size of $f^{-1}(B)$. He showed that if $R$ is a ring (with unity) and ${f_*}_n\co H_n(M;R)\to H_n(N;R)$ is surjective then the composite
$$
\check H^{i+k}(f^{-1}(B);\!R) \to H_{n-i}(M,M{-}f^{-1}(B);\!R) \stackrel{f_*}{\to} H_{n-1}(N,N{-}B;\!R) \to H^i(B;\!R)$$
is also surjective for all $i$. This holds for all maps in the same homotopy class of $f$.

In topological coincidence theory, one studies the coincidence set $$C(p,q)=\{x\in X\,|\,p(x)=q(x)\}$$ of two maps $p,q\co X\to Y$. Let $\Delta Y$ denote the diagonal of $Y$ in  $Y\times Y$. Then $C(p,q)=f^{-1}(B)$ with  
$M=X, N=Y\times Y$ and $B=\Delta Y$, and  $f\co M\to N$ is given by $f=(p\times q)\circ d$ where $d\co X \to X\times X$ is the diagonal. In \cite{GJW}, we studied the coincidence problem for transverse maps. Since $C(p',q')$ changes for $p'\sim p, q'\sim q$, it is natural to look for homological information about $C(p',q')$.

Motivated by this question in coincidence theory and by the work of Wyler \cite{wyler}, we explore the general problem of obtaining a (co)homological measure of the coincidence set $C(p,q)$ or the preimage set $f^{-1}(B)$ in general.

Our first objective of this paper is to give a general formula in the context of homology and cohomology with local coefficients which contains both Gottlieb's formula \eqref{gottlieb} and the formula \eqref{gjw-formula} as special cases, thereby further exploring the so-called {\it vague principle\/} of H\,Hopf. Secondly, we give homological conditions on a map $f\co M \to N \supset B$ such that $f^{-1}(B)$ has nontrivial cohomology with local coefficients.

This paper is organized as follows. In \fullref{sec2}, we give a generalized Gottlieb's theorem (\fullref{general}) with local coefficients. This formula reduces to Gottlieb's formula in the root case with trivial integral coefficients and orientable manifolds (\fullref{root}) and it reduces to the formula \eqref{gjw-formula} in the coincidence setting (\fullref{coin}). The formula in the root case does not seem to be a special case of that in the coincidence case unless further conditions are imposed (see \fullref{coin=root}). In \fullref{sec4}, we generalize the work of A\,Wyler \cite{wyler} in studying the nonvanishing of the induced homomorphism of a map with simple coefficients. In \fullref{sec5}, we obtain the results of \fullref{sec4} for local coefficients.

{\bf Acknowledgements}\qua This work was conducted during the first author's visit to Bates College, April 11--23, 2003 and the second 
author's visits to IME-USP, October 14--21, 2002, May 12--22, 2003, and April 27--May 4, 2004. The visit of the first author was supported by a grant 
from the ``Projeto tem\'atico Topologia
Alg\'ebrica e Geom\'etrica-FAPESP." The visits of the second author 
were partially supported by a grant 
from Bates College, the ``Projeto tem\'atico Topologia
Alg\'ebrica e Geom\'etrica-FAPESP," the ``Projeto 1-Pr\'o-Reitoria de Pesquisa-USP" and the NSF. We thank the referee for making useful suggestions which lead to a better and more coherent exposition.
\section{Generalized Gottlieb's formula}\label{sec2}

Let $M$, $N$ and $B \subset N$ be closed smooth compact  manifolds of dimensions $k{+}n, n$ and $\ell$, respectively.  Let $ f\co  M \to N$ be a map transverse to $B$ and $i_1\co  f^{-1}(B) \hookrightarrow M$ be the inclusion. Therefore, $F=f^{-1}(B)$ is a closed submanifold of dimension $k+\ell$. No orientability is assumed.

Recall that over 
a manifold $N$ there is a local coefficient system $\mathcal O_N$ called 
{\it the orientation local coefficient system\/} which has $\mathbb Z$ as the typical group and the 
homomorphism $w\co \pi_1(N) \to {\rm Aut}(\mathbb Z)$ is given by
$w(\alpha)=\sgn_N(\alpha)$, the sign of $\alpha$ in $N$. Consider a local coefficient system $R_N$ over $N$ with typical group $R$. Then we have the local system over $M$ induced by $f$, denoted by $R_M$, and define the local system $R_M'=R_M \bigotimes \mathcal O_M$.
Denote by $R_F'$ the local system over $F$ induced by $i_1\co F \hookrightarrow M$ from the system $R_M'$, ie $R'_F=i_1^*R'_M$.

We now define homomorphisms  $\phi(j)\co H^{j}(N,N{-}B;R_N) \to H_{n+k-j}(F; R'_F)$ for $0 \leq j \leq n$ and then we prove the general Gottlieb's formula.

\begin{definition}\label{Define} For $0\le j\le n$, let $\phi(j)$ be the composite homomorphism
$$ H^j(N, N{-}B;R_N)   \stackrel{ f^*}{\longrightarrow}  H^j(M,M-F;R_M)  \stackrel{ A^{-1}}{\longrightarrow} 
   H_{n+k-j}(F;R'_{F})$$
where A is the duality isomorphism as in Theorem 6.4 of Spanier \cite{Spanier} 
\end{definition}

A closer look at \fullref{Define} suggests that $\phi(j)$ can be expressed using the homology transfer as in Dold \cite[Definition 10.5, p.310]{dold}, provided we use local coefficients. This alternative way to describe $\phi(j)$ may have interest in its own right.

First we recall (for maps between oriented manifolds) that Dold's homology transfer is the composite
\begin{equation}\label{dold-transfer}
H_{n{-}j}(U_2) \stackrel{(\frown [z_N])^{-1}}{\longrightarrow} H^j(N,\!N{-}U_2) \stackrel{f^*}{\to} H^j(M,\!M{-}U_1) \stackrel{(\frown [z_M])}{\longrightarrow} H_{n{+}k{-}j}(U_1)
\end{equation}
where 
$U_2 \stackrel{i_2}{\hookrightarrow} N$ is an open set of $N$ and  $f^{-1}(U_2)=U_1$
By incorporating local coefficients without assuming orientability, we define a similar transfer $f_{!}$ to be the composite
\begin{multline}\label{homology-transfer}
H_{n-j}(U_2;i_2^*R'_N) \stackrel{D_N}{\longrightarrow} H^j(N,N-U_2; R_N) \stackrel{f^*}{\longrightarrow} H^j(M,M-U_1; R_M) \\
\stackrel{D^{-1}_M}{\longrightarrow} H_{n+k-j}(U_1;i_1^*R'_M)
\end{multline}
where $D_N$ and $D_M$ are the duality isomorphisms in $N$ and in $M$, respectively, with local coefficients and $R'_N=R_N\bigotimes \mathcal O_N$.

To relate the definition of $\phi(j)$ to the Dold transfer we introduce $\phi'(j)$ as follows.

\begin{definition}\label{Define1} 
Let $U \subset N$  be an open ($\bar U$ be a closed)  tubular neighborhood of $B$ such that $f^{-1}(U)=U_1 \subset M$ is  an open ($\bar U_1$ a closed) tubular neighborhood of 
$f^{-1}(B)=F$. Consider the composite 
\begin{equation}\label{alternative}
\phi'(j):= r_{1_*} \circ f_{!}\circ (r_{2_*})^{-1}\circ {D_N}^{-1}
\end{equation} 
where $r_2\co  U_2 \to B$ and $r_1\co  U_1 \to F$ are the retractions.
\end{definition}

 It can be shown that $\phi(j)=\phi'(j)$.

Next we compute the groups $H^{n-\ell}(N,N{-}B;R_N)$ and $H_{k+\ell}(F;R'_F)$
for an arbitrary system $R_N$, where $R'_F$ is obtained from $R'_N$ as above.
 Let $i_2\co B \to N$ be the inclusion. Write
\begin{equation*}
\begin{aligned}
R_1&=\{ r \in R \,|\, \sgn_N(\theta)\sgn_B(\theta)i_{\#}(\theta) r=r\} &\quad  
&\text{for all~} \theta \in \pi_1(B), \\
R_2&=\{ r \in R \,|\, \sgn_F(\alpha)\sgn_M(\alpha)f_{\#}(\alpha)r=r\} &\quad &\text{for all~} \alpha \in \pi_1(f^{-1}(B)).
\end{aligned}
\end{equation*}

\begin{proposition} \label{identification}
The groups $H^{n-\ell}(N,N{-}B;R_N)$ and $H_{k+\ell}(F,R'_F)$ are isomorphic to $R_1$ and $R_2$ respectively and $R_1\subset R_2$.
\end{proposition}     
\begin{proof} By duality with local coefficients \cite{Spanier}, $H^{n-\ell}(N,N{-}B;R_N)$ is
isomorphic to $H_{\ell}(B; R'_{\smash{B}})$ where $R'_{\smash{B}}=i_2^*R'_i{\smash{N}}$ is the restriction of $R'_N=R_N \bigotimes \mathcal O_N$ to $B$. Now we apply duality again in $B$ and we obtain
$H^0(B;R'_B \bigotimes \mathcal O_B)$ which is precisely the subgroup $R_1$. For the
calculation of 
$H_{k+\ell}(F;R'_F)$, again we use duality. The local coefficient system $R'_F \bigotimes \mathcal O_F$ is exactly the one given by $\alpha * r=\sgn_F(\alpha)\sgn_M(\alpha)f_{\#}(\alpha)r$.

In order to show that $R_1 \subset R_2$, observe that for $\alpha \in \pi_1(f^{-1}(B))$ we have that $\sgn_F(\alpha)\sgn_M(\alpha)=\sgn_N(f_{\#}(\alpha))\sgn_B(f_{\#}(\alpha))$. This follows since we have $U_2$ and $U_1$  tubular neighborhoods of $B$ and of $F$ respectively, where $U_1=f^{-1}(U_2)$. So the two actions coincide on the subset 
$\pi_1(f^{-1}(B))=\pi_1(F)$ and the result follows. 
\end{proof}

Now, we prove the following general Gottlieb's formula.
 
\begin{theorem} \label{general}
Let $f\co  M \to N$ be transverse to the submanifold $B \subset N$. Let 
$\phi(j)\co H^{j}(N,N{-}B;R_N) \to H_{n+k-j}(F; R'_F)$ be the homomorphism defined as a\-bove.
With respect to the local coefficient systems $R_N$, $R_M$, $R_M'$ and $R'_F$ over $N$, $M$, $M$ and $F$, respectively, for every $r \in H^j(N,N{-}B;R_N)$, the element  $f^*j_2^*(r)$ 
is the Poincar\'e dual of $i_{1_{*}}\phi(j)(r)$,
ie, 
\begin{equation}\label{general-formula}
i_{1_{*}}\phi(j)(r)=f^{*}(j_2^*(r))\frown [z_M].
\end{equation}
Furthermore, the homomorphism $\phi(n-\ell)$ corresponds to the inclusion after identifying 
$H^{n-\ell}(N,N{-}B;R_N)$ and $H_{k+\ell}(F;R'_F)$ with the subgroups $R_1$ and $R_2$ of $R$, respectively as in \fullref{identification}. 
\end{theorem}
\begin{proof}  
Consider the following commutative diagram where $j_2\co (N,\emptyset) \to (N,N{-}B)$ and $j_1\co (M,\emptyset) \to (M,M-F)$ are inclusions.
\begin{equation}\label{exact1}
\begin{CD}
    { M }   @>>>   N    \\
    @V{j_1}VV  @V{j_2}VV   \\
    {(M,M-F)}    @>>>  (N,N{-}B) 
\end{CD}
\end{equation}
Together with duality isomorphisms, \eqref{exact1} induces the commutative diagram:
\begin{equation}\label{exact2}
\begin{CD}
    H^j(N,\! N{-}B;\!R_N)   @>{f^*}>>  H^j(M,\!M{-}F;\!R_M) @>{A^{-1}}>>   H_{n+k-j}(F;\!R'_{F})  \\
    @V{j_{2}^*}VV  @V{{j_1}^*}VV   @V{i_{1_* }}VV \\
    {H^j(N;\!R_N)}    @>{f^*}>> H^j(M;\!R_{M}) @>{\frown [z_M]}>> H_{n+k-j}(M;\!R'_{M})
\end{CD}
\end{equation}
The composite homomorphism of the top row of \eqref{exact2} gives $\phi(j)$ by \fullref{Define}. Thus, the commutativity of \eqref{exact2} implies \eqref{general-formula}. Now, for $j=n-\ell$, \fullref{identification} completes the second assertion.
\end{proof}

\begin{remark} In \cite{dob2}, R\,Dobre{\'n}ko defined the primary obstruction $o_B(f)$ to deforming a map $f$ (transverse to $B$) out of a subspace $B\subset N$. In view of \fullref{general}, our general duality formula \eqref{general-formula} shows that Dobre{\'n}ko's obstruction class $o_B(f)$ is Poincar\'e dual to the image of the homology class representing $f^{-1}(B)$.
\end{remark}

\section{Roots and coincidences}

In this section we consider in details two special cases of \fullref{general}. For one of them, we take $B$ to be a point. This is related to the study of roots. The other case, we take $p,q\co  X\to Y$ a pair of maps
between two closed manifolds. Then we consider the map  $f\co M \to N$ where $f=(p\times q)\circ d, M=X, N=Y\times Y$ and $B=\Delta Y$. It turns out that $f^{-1}(B)=C(p,q)=\{x\in X\,|\,p(x)=q(x)\}$ is the coincidence set  of the two maps $p,q\co X\to Y$. 
This latter case corresponds to the study of coincidences of a pair of maps.

Suppose $B$ is a point and $R_N$ is an arbitrary local coefficient system over $N$ with typical group $R$. We obtain:

\begin{theorem}\label{root}
In the case $B=y_0$ is a point, the local coefficient system $R'_F$ is  $R\bigotimes \mathcal O_F$ and we have the following:
\begin{enumerate}
\item $H^j(N,N-y_0;R_N)=0$ for $j\ne n$ and $H^n(N,N-y_0;R_N)=R$.
\item $H_{k}(F;{R\bigotimes \mathcal O_F})=R$.
\item The homomorphism $\phi({n})\co H^n(N,N-y_0;R_N) \to H_k(F;R'_F)$ under the identification above is the identity. Therefore, if $y_0$ lies in the interior of an $n$--simplex $\sigma_0$ and $c_n$ denotes the elementary cocycle representing the cohomology fundamental class so that $c_n(\sigma_0)=r \in R$, then the pullback of $[c_n]$ is dual to the image of the homology class $[z_F]\otimes r$.
\item When $R_N=\mathbb Z$ is the trivial local coefficient system and the manifolds are closed and  orientable, we obtain Gottlieb's formula \eqref{gottlieb}.
\end{enumerate}
\end{theorem}
\begin{proof} (1)\qua By duality, $H^j(N,N-y_0;R_N) \cong H_{n-j}(y_0;R_{\{y_0\}})=0$ for $j\ne n$. On the other hand, $H_{0}(y_0;R_{\{y_0\}})=R$. 

(2)\qua The system $R'_F=\mathcal O_F \bigotimes R$  and $k=\dim F$ since $\ell=0$ so the assertion follows. 

(3)\qua This follows from \fullref{general}. 

(4)\qua When $R_N=\mathbb Z$ and $\ell=0$, $\phi(n)$ is the identity from (3) and $j_2$ induces an isomorphism identifying $H^n(N)$ with $H^n(N,N-y_0)$. Now \eqref{general-formula} reduces to Gottlieb's formula \eqref{gottlieb}. \end{proof} 

For the second case we consider  $f\co M \to N$ where $f$ is the map $(p\times q)\circ d$ from $M=X$ to $N=Y\times Y$, $B=\Delta Y$,  and  $F=(p\times q)^{-1}(\Delta Y)$ which is the same as $C(p,q)$.  
If  $\dim Y=n$ then we have $\dim N=2n$,
$\dim M=k{+}2n$ and $\ell=n$. Denote by $\psi\co \pi_1(N\times N) \to {\rm Aut}(R)$ the representation corresponding to the given local coefficient system. We obtain:

\begin{theorem}\label{coin} 
In the case $N=Y\times Y$ and $B=\Delta Y$, the local coefficient systems $R_{\Delta Y} (=R_B)$ and $R_{C(p,q)}(=R_F)$ are given by the equations $(\alpha,\alpha)*r=\psi(\alpha,\alpha)$ and $\beta *r=\psi(p_{\#}(\beta),q_{\#}(\beta))$, respectively. We have the following:
\begin{enumerate}
\item $H^j(Y^{\times};R_{Y{\times} Y})=0$ for $j<n=\dim Y$, and $H^n(Y^{\times};R_{Y{\times} Y})$ consists of the elements of $R$ that are fixed by the action of $\sgn_Y(\alpha)\psi(\alpha,\alpha)$ for all $\alpha\in \pi_1(Y)$. In particular, when $R_{Y{\times} Y}$ is given by $\pi_n(Y,Y{-}y_0)$, the group $H^n(Y^{\times};\!R_{Y{\times} Y})$ is isomorphic to the sum $\sum \mathbb Z$ indexed by the center of the group $\pi=\pi_1(Y)$.
\item $H_{k+n}(C(p,q);R'_{\smash{C(p,q)}})$ consists of the elements of $R$ that are fixed by the action of $\psi(p_{\#}\beta,p_{\#}(\beta))$ for all $\beta\in \pi_1(C(p,q))$. In particular, when $R_{Y\times Y}$ is given by $\pi_n(Y,Y-y_0)$, the group $H_{k+n}(C(p,q);R'_{\smash{C(p,q)}})$ is isomorphic to the sum $\sum \mathbb Z$ indexed by the centralizer of the subgroup $p_{\#}(\pi_1(C(p,q)))$ in $\pi_1(Y)$.
\item The homomorphism    
$$\phi({n})\co  H^n(Y^{\times};R_{Y\times Y}) \to H_{k+n}(C(p,q);R'_{C(p,q)})
$$
under the identification above is the inclusion. Thus, for $\sigma_0$ an $n$--simplex transversal to the diagonal and $c_n$ an elementary cocycle with $c_n(\sigma_0) = r \in R$, the pullback of $[c_n]$ is dual to the image of $[z_F] \otimes m$ where $[z_F]$ is the fundamental homology class of $F$.
\item When $R_N=\pi_n(Y,Y-y_0)$ and the manifolds are closed and orientable, we obtain the formula 
\eqref{gjw-formula}.
\end{enumerate}
\end{theorem}
\begin{proof} (1)\qua By duality, $H^j(Y^{\times}; R_{Y\times Y}) \cong H_{2n-j}(\Delta Y; R_{\Delta Y}')$. Since 
$\dim \Delta Y=n$, it follows that $H^j(Y^{\times};R_{\Delta Y})=0$ for $j<n$. When $j=n$, by duality again, we have $H_{n}(\Delta Y; R_{\Delta Y}')\cong H^0(\Delta Y; R_{\Delta Y})$ which is isomorphic to the submodule of $R_{\Delta Y}$ that is fixed by the action of $\psi$. This submodule in the case where $R_{Y\times Y}=\pi_n(Y,Y-y_0)$ is easily seen to be isomorphic to $\sum_\alpha \mathbb Z$ where $\alpha$ varies over the center of $\pi_1(Y)$. 

(2)\qua Note that $F=C(p,q)$ is a submanifold of dimension $k+n$. By duality, 
$H_{k+n}(F;R'_F) \cong H^0(F;R'_F\bigotimes \mathcal O_F)$. A similar argument as in (1) gives the result. 

(3)\qua This follows from \fullref{general}. 

(4)\qua The formula \eqref{general-formula} reduces to \eqref{gjw-formula} when $R_N=\pi_n(Y,Y-y_0)$.  
\end{proof} 

\begin{remark} The special cases above with local coefficients given by $\pi_n(Y,Y-y_0)$ are related to the study of obstructions in the root problem and in the general coincidence problem. While the root problem can sometimes be regarded as a special case of the coincidence problem, the formula \eqref{general-formula} in the root case as in \fullref{root} does not seem to follow from that in the general coincidence case as in \fullref{coin}.
\end{remark}

If the space $Y$ is a topological group, then by means of the map $h\co Y\times Y \to Y$ defined by $h(x,y)=xy^{-1}$, we obtain the formula for the root case from the general coincidence case as follows. 

\begin{proposition} \label{coin=root}
Given a topological group $Y$, a local coefficient system $R_Y$ and a map $f\co X \to Y$ which is transverse to the neutral element $e\in Y$, the general Gottlieb's formula for roots follows from that for the coincidence case.
\end{proposition}
\begin{proof} Consider the local coefficient system $R_{Y\times Y}$ over 
$Y\times Y$ induced by $h$ from the system $R_Y$ and the map $f\times \bar e\co X \to Y\times Y$ where $\bar e$ denotes the constant map at $e$. It follows that $F=f^{-1}(1)=(f\times \bar e)^{-1}(Y\times Y)$. Now, for a given class $x\in H^j(Y,Y-e;R_Y)$, we apply the formula \eqref{gjw-formula} to $h^*(x)\in H^n(Y^{\times}; 
R_{Y\times Y})$. Then, we obtain  $$i_{1_*}\phi(j)(h^*(x))=(f\times \bar e)^{*}(j_2^*(h^*(x)))\frown [z_X].$$ 
But $(f\times \bar e)^{*}(j^*(h^*(x))\frown [z_X]=f^*(j^*(x))\frown [z_X]$ since $f=h \circ (f\times \bar e)$ and it is straightforward to see that 
$$i_{1_*}\phi(j)(h^*(x))= j_{2_*}\phi(j)((x)).\proved$$
\end{proof}

\begin{remark}
The hypothesis that $Y$ be a group in \fullref{coin=root} can be slightly relaxed by requiring $Y$ to possess a multiplication with inverse. For example, \fullref{coin=root} holds if $Y=S^7$, the unit Cayley numbers.
\end{remark} 

\section[(Co)homological estimates for preimages -- simple coefficients]{(Co)homological estimates for preimages -- \\simple coefficients}\label{sec4}

Let $R$ be a ring with unity.
Suppose $M$ and $N$ are $R$--orientable manifolds of dimension $m$ and $n$ respectively, and $B\subset N$ is an $R$--orientable submanifold of dimension $k$. We study the relationship between $H^{*+m-n+k}(g^{-1}(B);R)$  and $H^*(B;R)$ (similarly in homology between $H_{*+m+n-k}(g^{-1}(B);R)$ and $H_*(B;R)$ ) for trivial coefficient systems, where $g$ is an arbitrary map belonging to $[f]$, the homotopy class of a map $f\co M \to N$. In particular, this will give lower bounds for the minimal dimension (cohomological, homological)
of $g^{-1}(B)$ among all possible maps $g\in[f]$.

\subsection[Bounds for the cohomology of the preimage of B under f]{Bounds for the cohomology of $f^{-1}(B)$}

Here, we extend the results of \cite{wyler}. 
Cech cohomology will be used. Since $B$ is a $k$--dimensional submanifold, we identify the Cech cohomology $\check H^*(B;R)$ with the singular cohomology $H^*(B;R)$ of $B$.

\begin{theorem}\label{Cohom-dim}
Let $R_0$ be the image of the homomorphism  
$f_{*_n}$ from $H_n(M;R)$ to $H_n(N;R)$. 
Let $i_0$ be the first integer (if it exists)  such that the induced 
homomorphism in homology  $f_{*_i}\co  H_i(M;R)\to H_i(N,N{-}B;R)$ is nonzero. Then 
\begin{enumerate}
\renewcommand{\labelenumi}{(\alph{enumi})}
\item If $i_0<n$ then $\check H^{m-i_0}(g^{-1}(B);R))\ne 0$ for any map $g$ homotopic to $f$.
\item If  $f_{*_n}$ is nontrivial then for $n-k \leq j\leq n$ the image of the homomorphism $f_{*_j}\co H_j(M,\!M{-}g^{-1}(B);\!R)\to H_j(N,\!N{-}B;\!R)$ contains $R_0\cdot H_j(N,\!N{-}B;\!R)$   and the image of the composite   
\begin{multline*}\hfil\check H^{m-j}(g^{-1}(B);R)\to H_j(M,M-g^{-1}(B);R)\hfil\\\to H_j(N,N{-}B;R) \to H^{n-j}(B;R)\end{multline*}
contains 
$R_0\cdot H^{n-j}(B;R)$. 
\end{enumerate}

In particular, $\check H^{m-n}(g^{-1}(B);R)$ and $\check H^{m-n+k}(g^{-1}(B);R)$ are nontrivial.
\end{theorem}
\begin{proof} 
(a)\qua First we have a commutative diagram
\begin{equation}\label{exact}
\begin{CD}
    { H_i(M;R) }   @>{f_*}>>   H_i(N;R)    \\
    @V{j_1}_*VV  @V{j_2}_*VV   \\
    {H_i(M,M-g^{-1}(B);R)}    @>{f_*}>>  H_i(N,N{-}B;R)
\end{CD}
\end{equation}
where $j_1, j_2$ are inclusions.

Since the composite $H_{i_0}(M;R) \to H_{i_0}(M,M-g^{-1}(B);R) \to H_{i_0}(N,N{-}B;R)$ is nonzero, it follows that $H_{i_0}(M,M-g^{-1}(B);R)$ is nontrivial. Now applying Poincar\'e Duality as in \cite{dold} 
using the fact that $g^{-1}(B)$ is compact, it follows that $\check H^{m-i_0}(g^{-1}(B);R)\ne 0$.

(b)\qua Given an element $\alpha \in H_j(N,N{-}B;R)$ and $r_0 \in R_0$, the element $r_0\cdot \alpha$ is the cap product of some class $\beta\in H^{n-j}(B;R)$ with the multiple $r_0 \cdot \mu_N$ of the homology fundamental class $\mu_N\in H_n(N;R)$. Using the naturality of the cap product 
\cite[p.239]{dold} and the fact that $r_0 \cdot \mu_N$ is in the image of $f_{*_n}$, we have 
$$r_0\cdot \alpha=\beta\frown r_0 \cdot \mu_N=\beta \frown (f_*(\psi))=f_*((f^*\beta)\frown \psi).$$ 
For $j=n-k$,
we have $H_{n-k}(N,N{-}B;R)$ isomorphic to $H^k(B;R)$ which  is isomorphic to
$R$. For $j=n$, we have $H_{n}(N,N{-}B;R)$ isomorphic to $H^0(B;R)$ which  is isomorphic to
$R$. In either case, $R_0 \cdot H^{n-j}(B;R)$ is nontrivial.
\end{proof} 

\begin{remark} In \fullref{Cohom-dim}, for any closed connected $B$, the image of 
$$f_{*_i}\co  H_i(M,M-g^{-1}(B);R) \to H_i(N,N{-}B;R)$$
is equal to $R_0 \cdot H_i(N,N{-}B;R)$. Furthermore, part (b) is similar to the result of A\,Wyler \cite{wyler} in that we obtain nontriviality of $f_{*_i}$ for all $i$ given the nontriviality of $f_{*_n}$.
\end{remark}

\subsection[Bounds for the homology of the preimage of B under f]{Bounds for the homology of $f^{-1}(B)$}

Next, we estimate the 
homology of $g^{-1}(B)$ for any $g$ in the homotopy class $[f]$ of $f$. We use Steenrod homology,  singular cohomology and Alexander--Spanier cohomology, denoted by $H_*$, $H^*$ and $\bar H^*$ respectively.
Our proof uses the same ingredients as in the last subsection. 

\begin{theorem}\label{Hom-dim}
Let $i_0$ be the first integer (if it exists) such that the induced 
homomorphism in cohomology $f^{\smash{{*^{i_0}}}}\co H^{i_0}(N,N{-}B;R)\to H^{i_0}(M;R)$ is a nonzero homomorphism. Then
\begin{enumerate}
\renewcommand{\labelenumi}{(\alph{enumi})}
\item If $i_0<n$ then $H_{m-i_0}(g^{-1}(B);R)\ne 0$ for any map $g$ homotopic to $f$.
\item If  $f^{*^n}\co H^n(N;R)\to H^n(M;R)$ is nontrivial then $f^{*^{n-k}}$ is also different from zero and $H_{m-n+k}(g^{-1}(B);R)\ne 0$.
\end{enumerate}
\end{theorem}
\begin{proof} 
Note that $H^i(N,N{-}B; R)$ is
zero for $i<n-k$ and for $i=n-k$ it is isomorphic to $H_k(B;R)\approx R$.

(a)\qua First we have a commutative diagram where the cohomology here is Alexander--Spanier cohomology.
\begin{equation}\label{exact'}
\begin{CD}
    {\bar H^i(N,N{-}B;R) }  @>{f^*}>>   \bar H^i(M,M-g^{-1}(B);R)    \\
    @V{j_2}^*VV  @V{j_1}^*VV   \\
    {\bar H^i(N;R)}    @>{f^*}>>  \bar H^i(M;R)
\end{CD}
\end{equation}
Since $M,N$ are compact and $B$ is a submanifold, we can identify $\bar H^i(N,N{-}B;R)$ and $H^i(M;R)$, the Alexander--Spanier cohomology, with their respective singular cohomology groups. Since the composite $$\bar H^{i_0}(N,N{-}B;R)\to \bar H^{i_0}(M,M-g^{-1}(B);R) \to \bar H^{i_0}(M;R)$$ is nonzero, it follows that $\bar H^{i_0}(M,M-g^{-1}(B);R)$ is nontrivial. Now applying Alexander--Spanier duality \cite[Theorem 11.15]{Massey} and using the fact that $g^{-1}(B)$ is compact, $H_{m-i_0}(g^{-1}(B);R)\ne 0$ and part (a) follows.

(b)\qua Using duality, the hypothesis on $f^{*^n}$ implies that the homomorphism $f_{*_n}$ from $H_n(M;R)$ to $H_n(N;R)$ is nontrivial. Again, we let $R_0=f_{*_n}(H_n(M;R))$ be the subring of $R$. Therefore, $r_0 \cdot \mu_N$ is in the image
of $f_{*_n}$ for any $r_0\in R_0$. As in the previous case, we use again the formula from \cite[p.239]{dold} for the naturality as follows.
Given an element $x'\in H^i(N,N{-}B;R)$, suppose that there exists $r_0\in R_0$ such that $r_0\cdot x'\ne 0$.
Thus 
$$r_0 \cdot x' \frown \mu_N=x'\frown r_0\cdot \mu_N=x' \frown f_{*n}(\beta)=f_*((f^*(x'))\frown \psi),$$ 
which then implies that $f^{*^i}(x')\ne 0$.
It follows that $f^{*^{n-k}}$ is nontrivial since we have $H^{n-k}(N,N{-}B;R)\approx R$. \end{proof} 

One natural way to use the results of the two subsections above is to find suitable coefficients $R$ such that our hypotheses hold. It may happen that for some ring the hypotheses do not hold but do hold for another ring. In the next section, we study the same situations as in this section when local coefficients are used. 
 
\begin{remark} Similar results can be obtained for $B$ a closed connected subset of $N$, using the same techniques above. However, one must use a suitable (co)homology in order to apply duality \cite[Theorems 11.11 and  11.15]{Massey} and replace $k=\dim B$ by the (co)homology dimension of the set $B$.
\end{remark}

\subsection{Roots and coincidences}\label{sec4.3}

We focus on two special cases, namely (i) when $B=\{b\}$ is a point (root case) and (ii) when $N=Y\times Y$, $B=\Delta Y$ and $f=p\times q$ where $p,q\co M\to Y$ with $\dim N=2n, \dim Y=n$ and $m\ge n$.

In situation (i), the homomorphism
$$
f_{*_i}\co H_i(M;R) \to H_i(N,N-b;R)$$
is always zero except possibly for $i=n$. Thus, \fullref{Cohom-dim} reduces to asserting that $H^{m-n}(g^{-1}(b);R)\ne 0$. Similarly, \fullref{Hom-dim} reduces to asserting that $\check H_{m-n}(g^{-1}(b);R)\ne 0$.

In fact, a similar phenomenon occurs in the coincidence case (ii) when simple coefficients are used.

Let $\tau_Y \in H^n(Y^{\times};R)$ be the Thom class of the normal bundle of the diagonal $\Delta Y$ in $Y\times Y$ where $Y^{\times}$ denotes the pair $(Y\times Y, Y\times Y -\Delta Y)$. Then the Thom Isomorphism Theorem asserts that every class $\beta \in H^{n+i}(Y^{\times};R)$ is given by $\beta =\tau_Y \smile {p_1}^*(\alpha)$ for some $\alpha \in H^i(Y;R)$ where $p_1\co Y\times Y \to Y$ denotes the projection on the first factor. Now, suppose the homomorphism
$$
H^n(Y^{\times};R) \stackrel{f^{*}}{\longrightarrow} H^n(M;R)$$
is trivial then for any $z\in H^{n+i}(Y^{\times};R)$, $z=\tau_Y\smile p_1^*(z')$ for some $z'\in H^i(Y;R)$. Now, $f^*(z)=f^*(\tau_Y) \smile f^*(z')=0$. For dimension reasons, $H^i(Y^{\times};R)=0$ for $i<n$. Hence, the first dimension $i$ at which 
$f^i\co H^*(Y^{\times};R) \to H^*(M;R)$ can possibly be nonzero is $*=n$. 

\begin{remark} The occurrence of the first nontrivial homomorphism $f^i$ at $i=n$ does not always hold true for arbitrary $B$ (other than the root or the coincidence cases). See for instance \fullref{nonprimary} even when one uses trivial coefficients.
\end{remark} 

\section[(Co)homological estimates for preimages -- local coefficients]{(Co)homological estimates for preimages -- \\local coefficients}\label{sec5}

In this section, we no longer require $M$, $N$ and $B$ to be orientable for we will use local coefficients.

Recall that in \fullref{gjw-ex}, the homomorphism $f^*$ is zero with trivial coefficients (integral coefficients) but nonzero for some local coefficients. In fact, using the universal coefficient theorem, the homomorphism $f^*$ is zero for any abelian group $G$ as simple coefficients. Although our results are stated in terms of arbitrary local coefficient systems, certain coefficient systems play an important role since they are 
closely related with the primary obstruction to deforming a map out of a subspace \cite{dob2}. 

Let us recall from Dobre{\'n}ko \cite{dob} a few useful results. We have that the pair $(N, N{-}B)$ is $(n-k-1)$--connected. The local system over $N$ given by the group $\pi_{n-k}(N,N{-}B)$ is computed as follows. 

\begin{theorem}{\rm \cite{dob}}\qua \label{dobrenko1}
Let $i\co B\to N$ be the inclusion and $\pi^0_1(B,b)$ be the subgroup of the elements of $\pi_1(B,b)$ which preserve the orientation of $B$. Then
\begin{enumerate}
\item The groups $\pi_j(N,N{-}B)$ vanish for $0\leq j < \dim N-\dim B=n-k$.
\item If $\mathrm{Ker} (i_{\#})\subset \pi_1^0(B,b)$ then $\pi_{n-k}(N,N{-}B)$ is isomorphic to a sum of $\mathbb Z$ indexed in $W(b)$, where $W(b)=\pi_1(N,b)/\mathrm{Im}(i_{\#}))$ (left cosets).
\item Otherwise, $\pi_{n-k}(N,N{-}B)$ is isomorphic to a sum of $\mathbb Z_2$ indexed in $W(b)$, where $W(b)=\pi_1(N,b)/\mathrm{Im}(i_{\#})$.
\item The action  of $\pi_1(N{-}B)$ on $\pi_{n-k}(N, N{-}B)$ is given by
$$\beta[\alpha]=\sgn_{N}(\beta)[\beta*\alpha]$$
for $\beta\in \pi_1(Y)=\pi_1(Y-B)$ and $\alpha\in  \pi_{n-k}(N, N{-}B)$.
\end{enumerate}
\end{theorem}

From the above result we can easily compute the first homology and cohomology of the pair $(N,N{-}B)$.

\begin{theorem}{\rm \cite{dob}}\qua\label{dobrenko2}
Using the local coefficient system $V_1 :=\pi_{n-k}(N,N{-}B)$ given by \fullref{dobrenko1},
\begin {enumerate}
\item $H^{n-k}(N,N{-}B;\pi_{n-k}(N,N{-}B))$ is equal to the direct sum of $R$ (for $R$ either $\mathbb Z$ or $\mathbb Z_2$) indexed over the set of elements of $\pi_1(N,b)/\mathrm{Im}(i_{\#})$ fixed by all elements of $\pi_1(B)$;
\item $H_{n-k}(N,N{-}B;\pi_{n-k}(Y,Y-B))$ is equal to the quotient of  $\pi_{n-k}(N,N{-}B)$ by the action of $\pi_1(B)$ (induced from the action of $\pi_1(N)$).
\end{enumerate}
\end{theorem}

The next two results are the local coefficients analogs of \fullref{Cohom-dim} and \fullref{Hom-dim}.

\subsection[Bounds for the cohomology of the preimage of B under f]{Bounds for the cohomology of $f^{-1}(B)$}

\begin{theorem}\label{local-Cohom-dim}
Let $R_N$ be a local coefficient system over $N$. Then $H_i(N,N{-}B; R_N)$ is
zero for $i<n-k$ and it is isomorphic to $H_0(B;R_N\bigotimes \mathcal O_N\bigotimes \mathcal O_B)\approx R/R_0$ for $i=n-k$ where $R_0$ denotes the image of the action.
Denote by $S^0_*$ the image of the homomorphism $f_*\co H_n(M;f^*(R'_N)) \to H_n(N;R'_N)=S^0$ where the ring $S^0=H^0(N;R_n)\cong H_n(N;R'_N)$ is the subring of elements of $R_N$ that are fixed by the action, and $R'_N=R_N\bigotimes \mathcal O_N$ is the induced local system corresponding to the orientation system $\mathcal O_N$. Let $i_0$ be the first integer (if it exists) such that the induced 
homomorphism in homology $f_*\co  H_*(M;f^*R_N) \to H_*(N,N{-}B;R_N)$ is nonzero. Then 
\begin{enumerate}
\renewcommand{\labelenumi}{(\alph{enumi})}
\item If $i_0<n$ then we have that $\check H^{m-i_0}(g^{-1}(B);f^*R_N)\ne 0$ for any map $g$ homotopic to $f$.
\item If $i_0=n$ then for $n-k \leq j\leq n$ the image of $$f_{*_j}\co H_j(M,M-g^{-1}(B);f^*R_N)\to H_j(N,N{-}B;R_N)$$ contains $S^0_*\cdot H_j(N,N{-}B;R_N)$ and the image of the composite   
\begin{multline*}\hfil\check H^{m-j}(g^{-1}(B);f^*(R'_N))\to H_j(M,M-g^{-1}(B);f^*R_N)\\ \to H_j(N,N{-}B;R_N) \to H^{n-j}(B;\iota^*(R'_N))\end{multline*}
contains 
$S^0_*\cdot H^{n-j}(B;\iota^*(R'_N))$ where $\iota\co B\to N$ is the inclusion. 
\end{enumerate}
In particular, $\check H^{m-n}(g^{-1}(B);f^*R_N)$ and $\check H^{m-n+k}(g^{-1}(B);f^*R_N)\ne 0$ are nontrivial.
\end{theorem}
\begin{proof} The proof is similar to the proof of \fullref{Cohom-dim} where we use duality with local coefficients as in \cite{Spanier}. 
\end{proof}

\subsection[Bounds for the homology of the preimage of B under f]{Bounds for the homology of $f^{-1}(B)$}
As in the trivial local coefficient case, $H_*$ means Steenrod homology.  We remark that for a compact space the homology  $\bar H_*$ used in \cite{Spanier} coincide with the Steenrod homology.
\begin{theorem}\label{local-Hom-dim}
Let $R_N$ be a local coefficient system over $N$. For $i<n-k$, we know that $H^i(N,N{-}B; R_N)$ is
zero. For $i=n-k$, it is isomorphic to $H^0(B;i_2^*(R_N))\approx R^0$ where $R^0$ denotes the set of elements of $R$ that are fixed by the action. Let $i_0$ be the first integer (if it exists) such that the induced 
homomorphism in cohomology $f^*\co H^i(N,N{-}B;R_N)\to H^i(M;f^*R_N)$ is nonzero. Here, $f^*R_N$ denotes the local system over $M$ by pulling back $R_N$ by $f$. Then 
\begin{enumerate}
\renewcommand{\labelenumi}{(\alph{enumi})}
\item If $i_0<n$ then $H_{m-i_0}(g^{-1}(B);f^*R_N)\ne 0$ for any map $g$ homotopic to $f$.
\item If $i_0=n$ then $f^{*^{n-k}}$ is also nonzero and  $H_{m-n+k}(g^{-1}(B);f^*R_N)\ne 0$.
\end{enumerate}
\end{theorem}     
\begin{proof} The proof is similar to the proof of \fullref{Hom-dim} where we use duality with local coefficients as in \cite{Spanier}. 
\end{proof}

Recall from \cite[Theorem 2.7, p.15]{dob2}, there is a universal element, which is the primary obstruction to deforming the identity map $1_N$ off the subspace $B$. We shall call this element the {\it Thom class\/} of $B$ in $N$ with coefficients in $R_U$, denoted by $\tau_B$. It was shown that for any map $f\co (M,M-f^{-1}(B)) \to (N,N{-}B)$, the primary obstruction ${o^{n-k}}_B(f)$ to deforming $f$ off the subspace $B$ on the $(n-k)$--th skeleton of $M$ is the pullback $f^*(\tau_B)$ of the Thom class where $k=\dim B$ and $n=\dim N$. 

The next result shows the relevancy of the local coefficient $R_U$. Since the pair $(N,N{-}B)$ is $(n-k-1)$--connected, the first possible dimension for which the homomorphism $f^*\co H^*(N,N{-}B;R_U) \to H^*(M;f^*(R_U))$ is nontrivial is $n-k$. Thus, for any local system $R_N$, we call the homomorphism $f^*\co H^{n-k}(N,N{-}B;R_N) \to H^{n-k}(M;f^*(R_N))$ the {\it primary\/} homomorphism.

\begin{proposition}\label{primary-homo} If the pullback of the Thom class in $H^{n-k}(N,N{-}B;R_U)$ by the map $f$ is trivial, then the primary homomorphism $f^*\co H^{n-k}(N,N{-}B;R_N) \to H^{n-k}(M;f^*(R_N))$ is trivial for any local coefficient system $R_N$. In particular the primary homomorphism with local coefficient $R_N$ is trivial.
\end{proposition}
\begin{proof} From the hypothesis, it follows that the primary obstruction to deforming $f$ into $N{-}B$ through the $(n-k)$--th skeleton is zero. So, up to homotopy, the restriction $f|_{M^{n-k}}\co M^{n-k} \to N$ of $f$ to the $(n-k)$--th skeleton  factors through $N{-}B$.

Note that the homomorphism $i^*\co H^{n-k}(M;i^*R_N) \to H^{n-k}(M^{n-k};R_N)$ in cohomology induced by the inclusion $i\co M^{n-k} \to M$ is injective for any local system of coefficients $R_N$. This follows from the cellular definition of cohomology with local coefficients as presented by Whitehead \cite{Whitehead}. For the 
cellular chains $C^j(\ \  ;R_N)$ for the two complexes $M$ and $M^{n-k}$ are isomorphic for  $j\leq n-k$, and 
$C^j(M^{n-k};R_N)$ is zero for $j>n-k$. So  
$H^{n-k}(M^{n-k};R_N)$ is a quotient of $C^{n-k}(M^{n-k};R_N)$ by the subgroup ${\rm Im}~\delta^{n-k-1}$ and $H^{n-k}(M;R_N)$ is the quotient of the $(n-k)$--cocycles contained in $C^{n-k}(M^{n-k};R_N)$ by the same subgroup ${\rm Im}~\delta^{n-k-1}$.
Hence the induced homomorphism by the composite map $M^{n-k} \to (N{-}B) \to N \to (N, N{-}B)$ is the zero homomorphism for any local coefficients because of the long exact sequence of the pair $(N,N{-}B)$ in cohomology. But this homomorphism is the composite of the primary homomorphism followed by $i^*$. Since $i^*$ is injective, the first assertion follows. The last part is clear from the first part.
\end{proof}

\begin{example} The primary homomorphism with local coefficient in general gives more information than the primary homomorphism with trivial coefficients (with $\mathbb Z$ coefficients for example). If we take the fiber map $p\co N_3 \to T^2$ from the three-dimensional Heisenberg manifold into the two-dimensional torus, this map has the property that the pullback of the fundamental class of $T^2$ with trivial $\mathbb Z$ coefficients is zero. However, the pullback of the twisted Thom class with local coefficients as given by Dobre{\'n}ko \cite{dob} is nonzero. See Gon{\c{c}}alves and Wong \cite{daci-peter2} for more details about this and other examples. 
\end{example}

\begin{example}\label{nonprimary} Let $n>2$, $M=N=S^{2n}$ and $B=S^n$. If $f$ is the identity then the primary homomorphism is trivial because $H^n(S^{2n};R_N)=0$ for any local coefficient system $R_N$. But the $2n$--th homomorphism is clearly the identity.
\end{example}

\begin{remark} In view of the proof of \fullref{primary-homo}, if $f$ is deformable into $N{-}B$ off the $(n-k+j)$--th skeleton of $M$ then the homomorphism $$f^*\co H^{n-k+j}(N,N{-}B;R_N) \to H^{n-k+j}(M;f^*(R_N))$$ is trivial for any $R_N$. This implies that the nonvanishing of the $(n-k+j)$--th homomorphism implies that there exist higher obstructions to deformation. In contrast to defining secondary or higher obstruction (sets), these higher homomorphisms are easily defined and therefore can be used as sufficient conditions for the existence of preimages.
\end{remark}

\begin{remark} For the coincidence case with simple coefficients in \fullref{sec4.3}, if the homomorphism $f^{*n}\co H^n(Y^{\times};R)\to H^n(M;R)$ is trivial then all the homomorphisms $f^{*k}$ are trivial for $k>n$. We do not know if the same phenomenon would hold true when local coefficients are used. In particular, we do not know whether the higher homomorphisms $f^{*k}$ would vanish if $f^{*n}\co H^n(Y^{\times}; \pi_n(N,N-b)) \to H^n(M;\pi_n(N,N-b))$ is trivial.
\end{remark}

\bibliographystyle{gtart}
\bibliography{link}

\end{document}